\documentclass{article}

\usepackage{arxiv}

\usepackage{geometry}
\usepackage{graphicx} %图像支持
\usepackage{amsmath} %数学公式
\usepackage{amsfonts}   %数学字体符号
\usepackage{amssymb}  %特殊符号
\usepackage{latexsym}  %特殊符号
\usepackage{float} %图片位置
\usepackage{verbatim}

\usepackage{float}     %浮动体，图片位置
\usepackage{epstopdf}   %插入eps格式图片
\usepackage{color}      %字体颜色
\usepackage{ntheorem}
\usepackage{mathrsfs}
\usepackage{hyperref}
\usepackage{indentfirst}
\usepackage{enumerate}
\usepackage{ntheorem}
\allowdisplaybreaks[4]

\theoremstyle{thmstyletwo}%
\newtheorem{theorem}{Theorem}%  meant for continuous numbers
\newtheorem{definition}{Definition}

\numberwithin{equation}{section}

\newtheorem{lemma}[theorem]{Lemma}

\newtheorem{assumption}[theorem]{Assumption}
\newtheorem{condition}[theorem]{Condition}
\newtheorem{example}[theorem]{Example}

\newtheorem{remark}[theorem]{Remark}

\newtheorem{proof}[theorem]{Proof}

\title{Strong convergence in the infinite horizon of numerical methods for stochastic differential equations}

\author{
 Wei Liu and Yudong Wang \\
 Department of Mathematics, Shanghai Normal University, Shanghai, 200234, China\\
 \texttt{weiliu@shnu.edu.cn; wangyudong19980120@gmail.com}
}

\begin{document}
\maketitle

\begin{abstract}
The strong convergence of numerical methods for stochastic differential equations (SDEs) for $t\in[0,\infty)$ is proved. The result is applicable to any one-step numerical methods with Markov property that have the finite time strong convergence and the uniformly bounded moment. In addition, the convergence of the numerical stationary distribution to the underlying one can be derived from this result. To demonstrate the application of this result, the strong convergence in the infinite horizon of the backward Euler-Maruyama method in the $L^p$ sense for some small $p\in (0,1)$ is proved for SDEs with super-linear coefficients, which is also a a standalone new result. Numerical simulations are provided to illustrate the theoretical results. 
\end{abstract}

% keywords can be removed

\section{Introduction}
The strong convergence in a finite time interval of numerical methods for stochastic differential equations (SDEs) has been an essential topic and attract lots of attention in the past decades. For any new proposed numerical methods of SDEs, the finite time strong convergence is always one of the fundamental properties to be investigated, for example the semi-implicit Euler-Maruyama method \cite{Hu1996}, the truncated Euler-Maruyama method \cite{LMY2021,Mao2015}, the tamed Euler method \cite{DKS2016,HJK2012}, and the fundamental mean-square finite time strong convergence theorem for any one-step method \cite{TZ2013}. Briefly speaking, for the solution of some SDE, $x(t)$, and its corresponding numerical solution, X(t), that is produced by some numerical method, the study of the strong convergence in some finite time interval seeks to find some upper bound of the difference between $x(t)$ and $X(t)$, i.e. for some positive constants $T$ and $p$
\begin{equation}\label{sec1:ftcon}
\sup_{0\leq t \leq T} \mathbb{E} |x(t) - X(t)|^p \leq C_T h,
\end{equation}
where $h$ is the step size and $C_T$ is a constant dependent on $T$. In most existing literatures, $C_T$ is an increasing function in terms of $T$, which means the above estimate of the error of the numerical method would become useless as $T \rightarrow \infty$.
\par
While, in this paper we try to obtain the error bound  for some constant $C$ which is independent from $T$, i.e.
\begin{equation}\label{sec1:ihcon}
\sup_{t \geq 0} \mathbb{E} |x(t) - X(t)|^p \leq C h,
\end{equation}
which is what the title of this paper indicates.
\par
A  naive question would be whether every numerical method that has the property \eqref{sec1:ftcon} possesses the property \eqref{sec1:ihcon} naturally. And the obvious answer is no. The geometric Brownian could be a quite illustrative example that if $ 2a+b^2 > 0$ then the numerical solution $X(t)$, for example generated by the EM method, to
\begin{equation*}
dx(t) = ax(t)dt + bx(t)dB(t)
\end{equation*}
satisfies \eqref{sec1:ftcon} but not \eqref{sec1:ihcon} for $p=2$. However, if the relation between $a$ and $b$ is changed into $2a+b^2 < 0$, we may have both  \eqref{sec1:ftcon} and  \eqref{sec1:ihcon} for $p=2$. In addition, for the case $2a + (p-1)b^2 <0$,  both  \eqref{sec1:ftcon} and  \eqref{sec1:ihcon} hold only for some small $p \in (0,1)$.
\par
Those observations from the geometric Brownian give us the hint that if the underlying solution has some moment boundedness property in the infinite horizon then there could be some numerical method that has the property \eqref{sec1:ftcon} as well as the property \eqref{sec1:ihcon}. And SDE models that naturally have such a boundedness property are not rarely found. Those SDE models that are used to describe to population, like the stochastic susceptible-infected-susceptible epidemic model \cite{GGHMP2011}, usually have their solutions naturally bounded within a finite area. In addition, we are also inspired by the above discussion that the variable $p$ may affect the results.
\par
This paper is devoted to study numerical methods for SDEs that are convergent strongly in the infinite horizon, i.e. \eqref{sec1:ihcon}. Since \eqref{sec1:ftcon} has been well studied  for many numerical methods, we do not want to waste those nice results in this paper. Hence, our strategy to prove \eqref{sec1:ihcon} for some numerical method is assuming that \eqref{sec1:ftcon}  is true together with some moment boundedness properties on the underlying and the numerical solutions.
\par
It is clear that such a strong convergence \eqref{sec1:ihcon} has its own importance. Moreover, the strong convergence in the infinite horizon may help to derive the same properties of numerical methods as those of underlying equations. For example, if the underlying SDE is stable in distribution, with the help of \eqref{sec1:ihcon}, we could conclude the numerical solution is also stable in distribution (see Section 2 for more detailed discussion). Other applications of such an infinite-horizon result contain approximating the mean exit time of some bounded domain of the underlying solution, as mentioned in Page 124 of \cite{HK2021} the finite time convergence does not apply for this problem. It should also be mentioned that such an error estimate for $t \in [0,\infty)$ of numerical methods for SDEs has been attracting increasing attention, see for example \cite{ACO2023} and \cite{CDO2021}.
\par
The main contribution of this paper is that we prove a new theorem that connects \eqref{sec1:ftcon} and  \eqref{sec1:ihcon} without giving much attention to the detailed structures of the numerical methods and the coefficients of the underlying SDEs.
\par
To demonstrate the application of the new theorem, we use the backward Euler-Maruyama (BEM) method as an example. For some small enough $p$, the BEM method is proved to be convergent in the infinite horizon for SDE with both superlinear drift and diffusion coefficients. And this is also a standalone new result, as in this paper the coefficient of the linear term in the drift coefficient is allowed to be positive, which is different from the existing result \cite{LMW2023}.
\par
This work is constructed in the following way. The main theorems and their proofs are put in Section 2. The strong convergence in the infinite horizon of the BEM method is proved in Section 3. Numerical results are displayed in Section 4. Section 5 sees the conclusion and future works.

\section{Assumptions and main results }
Throughout this paper, we let $(\Omega,\cal F,\mathbb{P})$ be a complete probability space with a filtration $\{{\cal F}_t\}_{t \geq 0}$ satisfying the usual conditions (that is, it is right continuous and ${\cal F}_0$ contains all $\mathbb{P}$-null sets). Let $W(t)$ be a scalar Brownian motion. Let $|\cdot|$ and $\langle \cdot \rangle$ denote the Euclidean norm and inner product in $\mathbb{R}^d$.

  In this paper, we consider the $d$-dimensional It\^{o} SDE
  \begin{equation}\label{SDE}
      dx(t)=f(x(t))dt+g(x(t))dW(t), \quad t \geq 0,
  \end{equation}
  with initial value $x(0)=x_0 \in \mathbb{R}^d$, where $f:\mathbb{R}^d \to \mathbb{R}^d$ and $g:\mathbb{R}^d \to \mathbb{R}^{d}$. In this paper, the underlying and the numerical solutions of \eqref{SDE} are sometimes defined as $x(t)$ and $X_k$, and sometimes to emphasize the initial value $x_0$, we also denote them by $x(t, x_0)$ and $X_k^{x_0}$. Now we list the following conditions.
\begin{condition} \label{as21}
  \begin{enumerate}
\item[(i)] There is a numerical method that can be used to approximate the solution of \eqref{SDE}, and the numerical solution $\{X_k\}_{k \geq 0}$ defined by this method is a time-homogeneous Markov process.
\item[(ii)]The underlying solution of \eqref{SDE} has attractive property, i.e., for some $p > 0$, and any compact subset $K \in \mathbb{R}^d$, any two underlying solutions with different initial values satisfy
\begin{equation*}
    \mathbb{E}\Big(|x(t,x_0)-x(t, y_0)|^p \Big) \leq \mathbb{E}\Big(|x_0-y_0|^p \Big)e^{-M_1t},
    \end{equation*}
    where $M_1$ is a positive constant and $(x_0, y_0) \in K \times K$.
\item[(iii)] In the finite time interval $[0, T]$, given a positive integer $n$, let $h=T/n$, then for some $p > 0$ consistent with $(ii)$ and any $j= 1, 2, \cdots , n$, the $p$th moment of the error of the numerical solution satisfies
\begin{equation*}
          \begin{aligned}
                 \mathbb{E}\Big(|x(jh) - X_{j}|^p \Big)
               \leq  C_T \Psi\Big(\mathbb{E}(|x(0)|^{r_1})\Big) h^q,
          \end{aligned}
\end{equation*}
where $C_T$ is a constant depends on T,  $q$ is a positive constant, $r_1$ is a nonnegative constant, and $\Psi (\cdot)$ is a continuous function.

\item[(iv)] The numerical solution is uniformly moment bounded, i.e., for some $r_1 > 0$ and any $k \in \mathbb{N} $
\begin{equation*}
   \mathbb{E}\Big(|X^{x_0}_k|^{r_1 }\Big) \leq M_2.
\end{equation*}
where $r_1$ is consistent with $(iii)$ and $M_2$ is a positive constant depend on $r_1$ and $x_0$.
\end{enumerate}
  \end{condition}

  \begin{theorem}\label{T22}
Suppose Condition \ref{as21} holds, then the numerical solution $\{X_k\}_{k \geq 0}$ is converges uniformly to the underlying solution of \eqref{SDE}, which means that for any fixed $h > 0$
 \begin{displaymath}
      \sup \limits_{j \in \mathbb{N}} \mathbb{E}|x(jh, x_0)-X_j^{x_0}|^p \leq C h^{q}
 \end{displaymath}
 where $C$ is a positive constant.
\end{theorem}
\begin{proof}
    First, choose a fixed $T>p\log2/c_1$. For any $i \in \mathbb{N}$, let $\overline{x(t,X_{in}^{x_0} )}$ denote a new underlying solution of \eqref{SDE} with the initial data  $X_{in}^{x_0}$. Then by the third and fourth assumptions of Condition \ref{as21},  we know that $\Psi$ is a continuous function with a finite domain, so that for any $x \in [0, M_2]$
\begin{equation}\label{T22-2}
    \Psi(x)  \leq M_3,
\end{equation}
where $M_3$ is a positive constant.
This implies that for any $jh \in [0,T]$
\begin{equation} \label{T22-3}
    \mathbb{E}|x(jh,x_0) - X_{j}^{x_0}|^p
            \leq  C_T \Psi\Big(\mathbb{E}(|x_0|^{r_1}) \Big) h^q \leq C_T M_3 h^q,
\end{equation}
 As for any $T+jh \in [T,2T]$, by the elementary inequality
\begin{equation}\label{ELE}
    (a+b)^p \leq (2(a\vee b))^p \leq 2^p (a^p+b^p),
\end{equation}
we derive that
\begin{equation} \label{T22-4}
    \begin{aligned}
        &\quad\mathbb{E}\Big(\Big|x(T+jh,x_0) - X_{n+j}^{x_0}\Big|^p\Big)  \\
        &\leq \mathbb{E}\Big(\Big|x(T+jh,x_0) -\overline{x(jh,X_{n}^{x_0})} + \overline{x(jh,X_{n}^{x_0})} - X_{n+j}^{x_0}\Big|^p\Big)   \\
        &\leq 2^p \mathbb{E}\Big(\Big|x(T+jh,x_0) -\overline{x(jh,X_{n}^{x_0})}\Big|^p\Big) + 2^p \mathbb{E}\Big(\Big|\overline{x(jh,X_{n}^{x_0})} - X_{n+j}^{x_0}\Big|^p\Big).
    \end{aligned}
\end{equation}
Since the underlying solution is time-homogeneous, then by \eqref{T22-3} and the second assumption of Condition \ref{as21}, we have
\begin{equation}\label{T22-5}
    \begin{aligned}
         \mathbb{E}\Big(\Big|x(T+jh,x_0) -\overline{x(jh,X_{n}^{x_0})}\Big|^p\Big)
         &=\mathbb{E}\Big(\Big|x(jh,x(T,x_0)) -\overline{x(jh,X_{n}^{x_0})}\Big|^p\Big) \\
        &\leq  \mathbb{E}\Big(\Big|x(T,x_0) - X_{n}^{x_0}\Big|^p\Big) e^{-M_1 jh} \\
        &\leq  e^{-M_1jh}C_T M_3 h^q.
    \end{aligned}
\end{equation}
Similarly, by the first and third assumptions of Condition \ref{as21} as well as \eqref{T22-2}, we get
\begin{equation}\label{T22-6}
    \begin{aligned}
        \mathbb{E}\Big(\Big|\overline{x(jh,X_{n}^{x_0})} - X_{n+j}^{x_0}\Big|^p\Big)
        &= \mathbb{E}\Big(\Big|\overline{x(jh,X_{n}^{x_0})} - X_{j}^{X_{n}^{x_0}}\Big|^p\Big)\\
        &\leq C_T \Psi\Big(\mathbb{E}(|X_{n}^{x_0}|^{r_1}) \Big) h^q  \\
        &\leq C_T M_3 h^q.
    \end{aligned}
\end{equation}
Inserting \eqref{T22-5} and \eqref{T22-6} into \eqref{T22-4} yields
\begin{equation}\label{T22-7}
    \begin{aligned}
        \mathbb{E}\Big(\Big|x(T+jh,x_0) - X_{n+j}^{x_0}\Big|^p\Big)
         \leq 2^p e^{-M_1jh}C_T M_3 h^q + 2^p C_T M_3 h^q.
    \end{aligned}
\end{equation}
Next,  since $T>p\log2/c_1$, let
\begin{equation*}
    \gamma :=2^p e^{-M_1T} < 1.
\end{equation*}
Continuing this approach, for any $2T+jh \in [2T, 3T]$
\begin{equation}\label{T22-8}
    \begin{aligned}
        &\quad\mathbb{E}\Big(\Big|x(2T+jh,x_0) - X_{2n+j}^{x_0}\Big|^p\Big)  \\
        &\leq 2^p \mathbb{E}\Big(\Big|x(2T+jh,x_0) -\overline{x(jh,X_{2n}^{x_0})}\Big|^p\Big) + 2^p \mathbb{E}\Big(\Big|\overline{x(jh,X_{2n}^{x_0})} - X_{2n+j}^{x_0}\Big|^p\Big).
    \end{aligned}
\end{equation}
By \eqref{T22-7} and the second assumption of Condition \ref{as21}
\begin{equation}\label{T22-9}
    \begin{aligned}
     \mathbb{E}\Big(\Big|x(2T+jh,x_0) -\overline{x(jh,X_{2n}^{x_0})}\Big|^p\Big)
       &\leq  \mathbb{E}\Big(\Big|x(2T,x_0) - X_{2n}^{x_0}\Big|^p\Big) e^{-M_1 jh} \\
        &\leq  e^{-M_1jh}\Big(\gamma C_T M_3 h^q + 2^p C_T M_3 h^q\Big).
    \end{aligned}
\end{equation}
By the first and third assumptions of Condition \ref{as21} as well as \eqref{T22-2} again, we have
\begin{equation} \label{T22-10}
    \begin{aligned}
        \mathbb{E}\Big(\Big|\overline{x(jh,X_{2n}^{x_0})} - X_{2n+j}^{x_0}\Big|^p\Big) &\leq C_T \Psi\Big(\mathbb{E}(|X_{2n}^{x_0}|^{r_1}) \Big) h^q  \\
        &\leq C_T M_3 h^q.
    \end{aligned}
\end{equation}
Inserting \eqref{T22-9} and \eqref{T22-10} into \eqref{T22-8}, we get
\begin{equation*}
    \begin{aligned}
        \mathbb{E}\Big(\Big|x(2T+jh,x_0) - X_{2n+j}^{x_0}\Big|^p\Big)
        \leq 2^p  e^{-M_1jh}\Big(\gamma C_T M_3 h^q + 2^p C_T M_3 h^q\Big) + 2^p C_T M_3 h^q.
    \end{aligned}
\end{equation*}
Similarly, for any $3T+jh \in [3T,4T]$, we have
\begin{align*}
         \mathbb{E}\Big(\Big|x(3T+jh,x_0) - X_{3n+j}^{x_0}\Big|^p\Big)
       \leq 2^p  e^{-M_1jh}\Big(\gamma^2 C_T M_3 h^q + \gamma 2^p C_T M_3 h^q + 2^p C_T M_3 h^q\Big)
        + 2^p C_T M_3 h^q.
\end{align*}
Now, we assume that for any $(m-1)T + jh \in [(m-1)T, mT]$, the following inequality holds
    \begin{align}\label{T22-11}
        &\quad  \mathbb{E}\Big(\Big|x((m-1)T+jh,x_0) - X_{(m-1)n+j}^{x_0}\Big|^p\Big) \\ \notag
        &\leq 2^p  e^{-M_1jh}\Big(\gamma^{m-2} C_T M_3 h^q +  2^p C_T M_3 h^q \times  \sum \limits_{i = 0}^{m-3} \gamma ^i\Big) + 2^p C_T M_3 h^q .
    \end{align}
Then for any $mT+jh \in [mT, (m+1)T]$,
    \begin{align}\label{T22-12}
         &\quad\mathbb{E}\Big(\Big|x(mT+jh,x_0) - X_{mn+j}^{x_0}\Big|^p\Big)  \\ \notag
        &\leq 2^p \mathbb{E}\Big(\Big|x(mT+jh,x_0) -\overline{x(jh,X_{mn}^{x_0})}\Big|^p\Big) + 2^p \mathbb{E}\Big(\Big|\overline{x(jh,X_{mn}^{x_0})} - X_{mn+j}^{x_0}\Big|^p\Big).
    \end{align}
By \eqref{T22-11} and Condition \ref{as21} as well as \eqref{T22-2}, we obtain
   \begin{align}\label{T22-13}
          &\quad\mathbb{E}\Big(\Big|x(mT+jh,x_0) -\overline{x(jh,X_{mn}^{x_0})}\Big|^p\Big) \\ \notag
          &=\mathbb{E}\Big(\Big|x(jh,x(mT,x_0)) -\overline{x(jh,X_{mn}^{x_0})}\Big|^p\Big) \\ \notag
        &\leq  \mathbb{E}\Big(\Big|x(mT,x_0) - X_{mn}^{x_0}\Big|^p\Big) e^{-M_1 jh} \\ \notag
        &\leq  e^{-M_1jh}\Big(\gamma^{m-1} C_T M_3 h^q +  2^p C_T M_3 h^q \times  \sum \limits_{i = 0}^{m-2} \gamma ^i\Big) , \notag
   \end{align}
and
    \begin{align}\label{T22-14}
         \mathbb{E}\Big(\Big|\overline{x(jh,X_{mn}^{x_0})} - X_{mn+j}^{x_0}\Big|^p\Big)
      &\leq C_T \Psi\Big(\mathbb{E}(|X_{mn}^{x_0}|^{r_1}) \Big) h^q  \\ \notag
        &\leq C_T M_3 h^q.
    \end{align}
Inserting \eqref{T22-13} and \eqref{T22-14} into \eqref{T22-12}, we have
\begin{equation*}
    \begin{aligned}
         &\quad\mathbb{E}\Big(\Big|x(mT+jh,x_0) - X_{mn+j}^{x_0}\Big|^p\Big)  \\
         &\leq 2^p  e^{-M_1jh}\Big(\gamma^{m-1} C_T M_3 h^q +  2^p C_T M_3 h^q \times  \sum \limits_{i = 0}^{m-2} \gamma ^i\Big) + 2^p C_T M_3 h^q.
         \end{aligned}
         \end{equation*}
         Since $ 0<\gamma<1 $, there exists a positive constant upper bound $M_4$ for the series $\sum \limits_{i = 0}^{\infty} \gamma ^i$. Then for any $m, k \in \mathbb{N}$, let
         \begin{equation*}
             C := 2^pC_T M_3 +2^{2p} C_T M_3 M_4 + 2^p C_T M_3.
         \end{equation*}
         The desired assertion follows.
\end{proof}
\begin{remark}
\begin{enumerate}
\item[(1)] The uniform convergence rate is consistent with the finite-time convergence rate.
\item[(2)] If the second assumption of Condition \ref{as21} is replaced by the  global attractivity of the numerical solution, \eqref{T22} still holds.
\end{enumerate}
\end{remark}
Next, we will discuss the connection between the uniform convergence and the stationary distribution.

Before we proceed, let us introduce some necessary notions about the stationary distribution. For any $x \in \mathbb{R}^d$ and any Borel set $\mathit{B} \subset \mathbb{R}^d$, the transition probability kernel of the underlying solution $x(t)$ and the numerical solution $X_k$ with initial value $x(0)=X_0=x_0$ is defined as
\begin{equation*}
    {\overline{\mathbb{P}}}_t(x_0,\mathit{B}) := \mathbb{P}(x(t) \in \mathit{B}|x(0) = x_0) \quad and  \quad \mathbb{P}_k(x_0,\mathit{B}) := \mathbb{P}(X_k \in \mathit{B}|X_0 = x_0).
\end{equation*}

Denote the family of all probability measures on $\mathbb{R}^d$ by  $\mathcal{P}(\mathbb{R}^d)$. Define by $\mathbb{L}$ the family of mappings $\mathit{F}$ : $\mathbb{R}^d \to \mathbb{R}$  satisfying
\begin{equation*}
    |\mathit{F}(x) - \mathit{F}(y)| \leq |x-y| \quad and \quad |\mathit{F}(x)| \leq 1,
\end{equation*}
for any $x,y \in \mathbb{R}^d$. For $\mathbb{P}_1, \mathbb{P}_2 \in \mathcal{P}(\mathbb{R}^d)$, define metric $\mathit{d}_{\mathbb{L}}$ by
\begin{equation*}
    \mathit{d}_{\mathbb{L}}(\mathbb{P}_1 , \mathbb{P}_2) = \sup \limits_{\mathit{F} \in \mathbb{L}} \left| \int_{\mathbb{R}^d} \mathit{F}(x) \mathbb{P}_1(dx) - \int_{\mathbb{R}^d} \mathit{F}(x)\mathbb{P}_2(dx) \right|.
\end{equation*}
The weak convergence of probability measures can be illustrated in terms of metric $\mathit{d}_{\mathbb{L}}$ \cite{IW1989}. That is, a sequence of probability measures $\{\mathbb{P}_k\}_{k \geq 1}$ in $\mathcal{P}(\mathbb{R}^d)$ converge weakly to a probability measure $\mathbb{P} \in \mathcal{P}(\mathbb{R}^d)$ if and only if
\begin{equation*}
    \lim \limits_{ k \to \infty} \mathit{d}_{\mathbb{L}}(\mathbb{P}_k, \mathbb{P}) = 0.
\end{equation*}
Then we define the stationary distribution for the underlying solution of \eqref{SDE} by using the concept of weak convergence.

\begin{definition}
    For any initial value $x \in \mathbb{R}^d$, the underlying solution of \eqref{SDE} is said to have a stationary distribution $\pi \in \mathcal{P}(\mathbb{R}^d)$ if the transition probability measure ${\overline{\mathbb{P}}}_t(x,\cdot)$ converges weakly to $\pi( \cdot)$ as $t \to \infty$ for every $x \in \mathbb{R}^d$, that is
    \begin{equation*}
        \lim \limits_{k \to \infty} \left(\sup \limits_{\mathit{F} \in \mathbb{L}} \left|\mathbb{E}(\mathit{F}(x(t)))-\mathbb{E}_{\pi}(\mathit{F})\right|\right) = 0,
    \end{equation*}
    where
    \begin{equation*}
        \mathbb{E}_{\pi}(\mathit{F}) = \int_{\mathbb{R}^d} \mathit{F}(y)\pi(dy).
    \end{equation*}
\end{definition}
If we add an additional condition.
\begin{condition}\label{as22}
    The underlying solution is uniformly moment bounded, i.e., for any $t \geq 0 $ and some $p > 0$ consistent with the second assumption of Condition \ref{as21}
\begin{equation*}
   \mathbb{E}\Big(|x(t,x_0)|^{p}\Big) \leq M_5,
\end{equation*}
where $M_5$ is a positive constant depend on $p$ and $x_0$.
\end{condition}

Then, from Theorem 3.1 in \cite{YM2003}, we know that the solution of \eqref{SDE} has a unique stationary distribution denoted by $\pi (\cdot)$ under Conditions \ref{as21} and \ref{as22}.

Thus, we have the following theorem.
\begin{theorem}\label{T25}
Suppose Conditions \ref{as21} and \ref{as22} hold.
   Then, the probability measure of the numerical solution converges to the stationary distribution of the underlying solution, that is
   \begin{equation*}
       \lim \limits_{ jh \to \infty \atop h \to 0} \mathit{d}_{\mathbb{L}}(\mathbb{P}_j(x_0,\cdot), \pi(\cdot)) = 0.
   \end{equation*}
\end{theorem}
\begin{proof}
    First, since the underlying solution of \eqref{SDE} has a unique stationary distribution $\pi$, this means that for any $\epsilon >0$, there exists an $T_1 >0$ such that for any $jh \geq T_1$
    \begin{equation*}
        \mathit{d}_{\mathbb{L}}(\overline{\mathbb{P}}_{jh}(x_0, \cdot), \pi(\cdot)) \leq \frac{\epsilon}{2}.
    \end{equation*}
    Next, from the definition of $\mathbb{L}$, for any $F \in \mathbb{L}$, we can get
    \begin{equation}\label{T25-15}
        |\mathbb{E}(F(x(jh))) - \mathbb{E}(F(X_j))| \leq \mathbb{E}(2 \wedge |x(jh) - X_j|).
    \end{equation}
    Since the numerical solution satisfies \eqref{T22}, choose $h$ sufficiently small such that $Ch^q \leq \min\{\frac{\epsilon}{8}, \frac{\epsilon^p}{2^p}\}$. Then, if $p \geq 1$, we have
    \begin{equation}\label{T25-16}
        \mathbb{E}(2 \wedge |x(jh) - X_j|) \leq \mathbb{E}(|x(jh) - X_j|)\leq (\mathbb{E} |x(jh) - X_j|^p)^{\frac{1}{p}} \leq \frac{\epsilon}{2},
    \end{equation}
    and if $p \in (0, 1)$, we have
    \begin{equation}\label{T25-17}
       \begin{aligned}
            \mathbb{E}(2 \wedge |x(jh) - X_j|) &\leq 2\mathbb{P}(|x(jh) - X_j| \geq 2) + \mathbb{E}(I_{\{|x(jh) - X_j| <2\}}|x(jh) - X_j|) \\
            &\leq 2^{1-p}\mathbb{E} |x(jh) - X_j|^p + \mathbb{E}(2^{1-p} |x(jh) - X_j|^p) \\
            & \leq 2^{2-p} \mathbb{E} |x(jh) - X_j|^p \\
            &\leq \frac{\epsilon}{2}.
       \end{aligned}
    \end{equation}
    Hence, it follows from \eqref{T25-15}, \eqref{T25-16}, \eqref{T25-17} that
    \begin{equation*}
    |\mathbb{E}(F(x(jh))) - \mathbb{E}(F(X_j))| \leq \frac{\epsilon}{2}.
    \end{equation*}
   Consequently, it is obvious that
    \begin{equation*}
         \sup \limits_{\mathit{F} \in \mathbb{L}}|\mathbb{E}(F(x(jh))) - \mathbb{E}(F(X_j))| \leq \frac{\epsilon}{2},
    \end{equation*}
    that is
    \begin{equation*}
        \mathit{d}_{\mathbb{L}}(\mathbb{P}_j(x_0,\cdot), \overline{\mathbb{P}}_{jh}(x_0, \cdot)) \leq \frac{\epsilon}{2}.
    \end{equation*}
   Then, the triangle inequality yields
   \begin{equation*}
       \mathit{d}_{\mathbb{L}}(\mathbb{P}_j(x_0,\cdot), \pi(\cdot~) \leq \epsilon.
   \end{equation*}
   The proof is hence completed.
\end{proof}

\section{The BEM method}
In this section, the BEM method is used as an example. Under some conditions that are weaker than the existing results, we not only obtain the uniform convergence of the BEM method but also prove that it can be used to numerically approximate the stationary distribution of the underlying solution. Now we make the following assumptions.
\begin{assumption}\label{as31}
     There exists a pair of constants $q \in [1,\infty)$ and $L_1 \in (0,\infty)$ such that
     \begin{equation}\label{S3-1}
         |f(x_1)-f(x_2)| \leq L_1(1+|x_1|^{q-1}+|x_2|^{q-1})|x_1-x_2|
     \end{equation}
     for all $x_1,x_2 \in \mathbb{R}^d$.
  \end{assumption}

    \begin{assumption}\label{as32}
      There exists  $c_1, c_2, c_3\in (0,\infty)$ and $l_1 \geq \max\{2q, 3\}, l_2 \geq 3 $ such that
      \begin{equation}\label{S3-2}
     2\langle x,f(x) \rangle +l_1|g(x)|^2 \leq c_1|x|^2 + c_2
 \end{equation}
        \begin{equation}\label{S3-3}
            \begin{aligned}
   2\langle x_1-x_2,f(x_1)-f(x_2) \rangle + l_2|g(x_1)-g(x_2)|^2\leq c_3|x_1-x_2|^2
\end{aligned}
        \end{equation}
for all $x, x_1, x_2 \in \mathbb{R}^d$.
    \end{assumption}

 From Assumptions \ref{as31} and \ref{as32}, we can get the following result \cite{AK2017}:
 \begin{equation}\label{S3-4}
     |g(x_1)-g(x_2)| \leq L_2(1+|x_1|^{q-1}+|x_2|^{q-1})|x_1-x_2|
 \end{equation}
 for all $x_1, x_2 \in \mathbb{R}^d$, where $L_2$ is a positive constant depends on $L_1$ and $l_2$.
 From \eqref{S3-1} and \eqref{S3-4}, we further deduce the following polynomial growth bound
 \begin{equation}\label{LM3-5}
     |f(x)| \vee |g(x)| \leq  L_3 (1 + |x|^q),
 \end{equation}
 where $L_3$ is a positive constant depends on $L_1, L_2, f(0)$ and $g(0)$.
 %for all $u \in \mathbb{R}^d$.
 This means that under the above assumptions, the solution of \eqref{SDE} is uniquely determined \cite{M2007}.
 The BEM method applied to \eqref{SDE} produces approximations $X_k \approx x(kh) $ by setting $X_0=x(0)=x_0$ and forming
 \begin{equation} \label{DASDE}
     X_k=X_{k-1}+f(X_k)h+g(X_{k-1}) \Delta W_{k-1},
 \end{equation}
 where $h > 0$ is the timestep and $\Delta W_{k-1} :=W(kh)-W((k-1)h)$ is the Brownian increment.
 %\begin{lemma}
   %  Under Assumption \eqref{as22} and let $c_1h < 1$, the BEM method \eqref{DASDE} is well defined.
 %\end{lemma}
%We refer the readers to  \cite{MS2013} for the proof.

 We point out that the BEM method \eqref{DASDE} is well-defined under  \eqref{as32} (see, e.g., \cite{CG2020}). And following the same argument of Theorem 2.7 in \cite{LM2015}, we get the following result.
\begin{lemma}\label{L33}
    The BEM method \eqref{DASDE} is a homogeneous Markov process.
\end{lemma}

    \subsection{The uniform moment boundedness.}
    In this subsection, for proof, we need some additional assumptions on diffusion coefficient $g$.
    \begin{assumption} \label{as33}
    There  exist three constants $\alpha > 0$, $D > 0$, and $C_3 \geq 0$ such that for any $x\in \mathbb{R}^d$
    \begin{equation*}
         \frac{(1-l_1)|g(x)|^2}{D+|x|^2+\alpha|g(x)|^2}-\frac{2|\langle x,g(x) \rangle|^2}{(D+|x|^2+\alpha|g(x)|^2)^2} \leq k_1 + \frac{P(x)}{(D+|x|^2+\alpha|g(x)|^2)^2} ,
    \end{equation*}
    where $k_1$ is a constant with $k_1+c_1 < 0$ and $P(x)$ is a polynomial of $x$ that satisfies the following condition for a sufficiently small constant $p$
    \begin{equation}
        \sup \limits_{x\in \mathbb{R}^d}\left| \frac{P(x)}{(D+|x|^2+\alpha|g(x)|^2)^{2-\frac{p}{2}}}\right| \leq C_3.
    \end{equation}

\end{assumption}

\begin{assumption} \label{as34}
    There exist a  constant $\beta > 0$ such that for any $x$, $y \in \mathbb{R}^d $
    \begin{equation}
    \frac{(1-l_2)|g(x)-g(y)|^2}{|x-y|^2+\beta|g(x)-g(y)|^2}-\frac{2 |\langle x-y,g(x)-g(y) \rangle|^2}{(|x-y|^2+\beta|g(x)-g(y)|^2)^2} \leq k_2,
\end{equation}
where  $k_2$ is a constant with $k_2+c_3 < 0$ and $x \neq y$.
\end{assumption}
\begin{remark}
We emphasize that the family of drift and diffusion functions that satisfy \eqref{as31} - \eqref{as34} is large. For example, for any polynomial $g(x) = a_0 + \sum \limits_{i=0}^n a_{i} x^{2i+1}$ with $a_i > 0$ and $n \geq 0$, if we choose $\alpha $ and $\beta$ sufficiently small, then $k_1$ and $k_2$ will be very close to $-a_1^2 (l_1 + 1)$ and $- a_1^2 (l_2 + 1)$, and $P (x)$ is usually a polynomial of degree less than $(D+|x|^2+\alpha|g(x)|^2)^{2-\frac{p}{2}}$. We can choose $c_1$ and $c_3$ that are less than $a_1^2 (l_1 + 1)$ and $ a_1^2 (l_2 + 1)$. Then for this $g(x)$, it is not difficult to see that \eqref{as33} and \eqref{as34} are satisfied and there are many  $f(x)$ satisfying Assumptions \ref{as31} and \ref{as32}.
\end{remark}

 \begin{lemma}\label{LM310}
    Suppose \eqref{as32} and \eqref{as33} hold, then there exists a pair of constants $(p^*, h^*)$ with $p^* \in (0, 1)$ and $h \in (0, 1)$ such that for any $p \in (0, p^*]$, $h \in (0, h^*]$ and $D \in (0,\infty)$, the BEM solution \eqref{DASDE} satisfies
 \begin{equation*}
     \mathbb{E}|X_k|^p \leq \mathbb{E}\Big(D+|x_0|^2+l_1h|g(x_0)|^2 \Big)^{p/2}-\frac{4M}{p(k_1+c_1)},
 \end{equation*}
 where $M$ is a  positive constant depends on $D$, $c_2$ and $p$.
\end{lemma}
\begin{proof}
 First, by  \eqref{DASDE} and \eqref{as32} , we have
\begin{equation}\label{LM310-26}
    \begin{aligned}
    &\quad|X_k|^2-|X_{k-1}|^2+|X_k-X_{k-1}|^2\\
        &=2 \langle X_k-X_{k-1},X_k \rangle \\
    &=2 \langle f(X_k),X_k \rangle h + 2 \langle g(X_{k-1}) \Delta W_{k-1},X_k \rangle \\
    &=2 \langle f(X_k),X_k \rangle h + 2 \langle g(X_{k-1}) \Delta W_{k-1},X_k -X_{k-1}\rangle +  2 \langle g(X_{k-1}) \Delta W_{k-1},X_{k-1} \rangle\\
    &\leq  (c_2+c_1|X_k|^2)h - l_1|g(X_k)|^2 h+|g(X_{k-1}) \Delta W_{k-1}|^2+|X_k-X_{k-1}|^2 +2 \langle g(X_{k-1}) \Delta W_{k-1},X_{k-1} \rangle.
    \end{aligned}
\end{equation}
Canceling the same terms on both sides gives
\begin{equation}\label{LM310-27}
        \begin{aligned}
        &(1-c_1h)|X_k|^2+l_1|g(X_k)|^2 h \\
        \leq &|X_{k-1}|^2+|g(X_{k-1}) \Delta W_{k-1}|^2
        +2 \langle g(X_{k-1}) \Delta W_{k-1},X_{k-1} \rangle +c_2 h.
        \end{aligned}
    \end{equation}
    Since $l_1 \geq 3$ and $c_1 > 0$, we see that for any constant $D > 0$
    \begin{equation}\label{LM310-28}
        (1-c_1h)(D+|X_k|^2+l_1|g(X_k)|^2 h ) \leq (D+|X_{k-1}|^2+l_1|g(X_{k-1})|^2 h )(1+\xi _{k-1}),
    \end{equation}
   where
    \begin{equation}\label{LM310-29}
        \begin{aligned}
         \xi _{k-1}=\frac{|g(X_{k-1}) \Delta W_{k-1}|^2-l_1|g(X_{k-1})|^2 h    +2 \langle g(X_{k-1}) \Delta W_{k-1},X_{k-1} \rangle +c_2 h}{D+|X_{k-1}|^2+l_1|g(X_{k-1})|^2 h }.
        \end{aligned}
    \end{equation}
     It is obvious that $\xi_{k-1}>-1$. Then we take conditional expectations with respect to $\mathcal{F}_{(k-1)h}$ on \eqref{LM310-28} leads to
  \begin{align}\label{LM310-30}
       &\mathbb{E}\Big((D+|X_{k}|^2+l_1|g(X_{k})|^2 h)^{p/2}\Big|\mathcal{F}_{(k-1)h}\Big) \\ \notag
        \leq  & \Big(1-c_1h\Big)^{-p/2}\Big(D+|X_{k-1}|^2+l_1|g(X_{k-1})|^2h\Big)^{p/2}
        \mathbb{E}\Big( (1+\xi _{k-1})^{p/2}\Big | \mathcal{F}_{(k-1)h}\Big) \\ \notag
        \leq & \Big(1-c_1h\Big)^{-p/2}\Big(D+|X_{k-1}|^2+l_1|g(X_{k-1})|^2h\Big)^{p/2}\mathbb{E}\Big( 1+\frac{p}{2} \xi _{k-1}+\frac{p(p-2)}{8}\xi _{k-1}^2  \\ \notag
        &+\frac{p(p-2)(p-4)}{2^3 \times 3!}\xi _{k-1}^3\Big | \mathcal{F}_{(k-1)h}\Big), \notag
  \end{align}
    where the last step, we use the following inequality
\begin{equation}\label{LM310-31}
    (1+u)^{p/2} \leq 1+\frac{p}{2}u+\frac{p(p-2)}{8}u^2+\frac{p(p-2)(p-4)}{2^3 \times 3!}u^3, \quad \forall p \in(0, 1),u\in (-1,\infty).
\end{equation}
    %for all $p \in (0, 1)$ and $u \in (-1, \infty)$.
    Since $\Delta W_{k-1}$ is independent of $\mathcal{F}_{(k-1)h}$,  for any $i \in \mathbb{N}^+,$  it is not difficult to see that
    \[\mathbb{E}\Big((\Delta W_{k-1})^{2i-1}\big|\mathcal{F}_{(k-1)h}\Big)=\mathbb{E}\Big((\Delta W_{k-1})^{2i-1}\Big)=0,\]
    \[\mathbb{E}\Big(|\Delta W_{k-1}|^{2i}\big|\mathcal{F}_{(k-1)h}\Big)=\mathbb{E}\Big(|\Delta W_{k-1}|^{2i}\Big)=(2i-1)!! h^i.\]
    This, together with \eqref{LM310-29} yields
    \begin{align} \label{LM310-32}
        &\mathbb{E}\Big(\xi _{k-1}\Big | \mathcal{F}_{(k-1)h}\Big) \\ \notag
     =&\mathbb{E}\Bigg(\frac{|g(X_{k-1}) \Delta W_{k-1}|^2-l_1|g(X_{k-1})|^2 h    +2 \langle g(X_{k-1}) \Delta W_{k-1},X_{k-1} \rangle +c_2 h}{D+|X_{k-1}|^2+l_1|g(X_{k-1})|^2 h }\Big | \mathcal{F}_{(k-1)h}\Bigg) \\ \notag
     =&\frac{(1-l_1)|g(X_{k-1})|^2 h +c_2h}{D+|X_{k-1}|^2+l_1|g(X_{k-1})|^2 h}.  \notag
    \end{align}
        Similarly, we can get
       \begin{align} \label{LM310-33}
             &\quad\mathbb{E}\Big(\xi _{k-1}^2\Big | \mathcal{F}_{(k-1)h}\Big) \\ \notag
                 %=&\mathbb{E}\Big(\frac{(3-2l_1+l_1^2)|g(X_{k-1})|^4 h^2 +4|\langle g(X_{k-1}), X_{k-1} \rangle|^2h + 2(1-l1)c_2|g(X_{k-1})|^2 h^2+c_2^2h^2}{(D+|X_{k-1}|^2+l_1|g(X_{k-1})|^2 h )^2}\Big | \mathcal{F}_{(k-1)h}\Big) \\
                 &=\mathbb{E}\Bigg(\Big(D+|X_{k-1}|^2+l_1|g(X_{k-1})|^2 h\Big)^{-2} \Big((3-2l_1+l_1^2)|g(X_{k-1})|^4 h^2 \\\notag
                 &\quad+4|\langle g(X_{k-1})\Delta W_{k-1}, X_{k-1} \rangle|^2
                 + 2(1-l_1)c_2|g(X_{k-1})|^2 h^2+c_2^2h^2\Big)\Big | \mathcal{F}_{(k-1)h}\Bigg) \\ \notag
                  &=\mathbb{E}\Bigg(\Big(D+|X_{k-1}|^2+l_1|g(X_{k-1})|^2 h \Big)^{-2} \Big(2|g(X_{k-1})|^4 h^2 + (1-l_1)^2|g(X_{k-1})|^4 h^2  \\ \notag
                  &\quad+ 2(1-l_1)c_2|g(X_{k-1})|^2 h^2+c_2^2h^2
                   +4|\langle g(X_{k-1})\Delta W_{k-1}, X_{k-1} \rangle|^2\Big)\Big | \mathcal{F}_{(k-1)h}\Bigg) \\ \notag
               % = &\mathbb{E}\Big(\frac{2|g(X_{k-1})|^4 h^2 + (1-l_1)^2|g(X_{k-1})|^4 h^2 + 2(1-l1)c_2|g(X_{k-1})|^2 h^2+c_2^2h^2}{(D+|X_{k-1}|^2+l_1|g(X_{k-1})|^2 h )^2} \\
               % & \quad \frac{+4|\langle g(X_{k-1}), X_{k-1} \rangle|^2h }{}\Big | \mathcal{F}_{(k-1)h}\Big) \\
                 &\geq  \frac{4|\langle g(X_{k-1}), X_{k-1} \rangle|^2h}{(D+|X_{k-1}|^2+l_1|g(X_{k-1})|^2 h )^2} \notag
       \end{align}
        and
            \begin{align}\label{LM310-34}
               & \quad  \mathbb{E}\Big(\xi _{k-1}^3\Big | \mathcal{F}_{(k-1)h}\Big) \\ \notag
                 &=\mathbb{E}\Bigg(\frac{\big[|g(X_{k-1}) \Delta W_{k-1}|^2-l_1|g(X_{k-1})|^2 h  }{(D+|X_{k-1}|^2+l_1|g(X_{k-1})|^2 h )^3} \frac{ +2 \langle g(X_{k-1}) \Delta W_{k-1},X_{k-1} \rangle +c_2 h \big]^3}{}\Big | \mathcal{F}_{(k-1)h}\Bigg) \\ \notag
                 &:= B_1+B_2+B_3+B_4. \notag
                 %\leq& \frac{c_2^3h^3}{(D+|X_{k-1}|^2+l_1|g(X_{k-1})|^2 h)^3}.
            \end{align}
        In the sequel, we will estimate $B_1, B_2, B_3, B_4$ separately. Firstly, we have
            \begin{align}\label{LM310-35}
                 B_1 &=\mathbb{E}\Bigg(\frac{\big[|g(X_{k-1}) \Delta W_{k-1}|^2-l_1|g(X_{k-1})|^2 h    +2 \langle g(X_{k-1}) \Delta W_{k-1},X_{k-1} \rangle \big]^3}{(D+|X_{k-1}|^2+l_1|g(X_{k-1})|^2 h )^3}\Big | \mathcal{F}_{(k-1)h}\Bigg) \\  \notag
                & = \frac{(15-9l_1 + 3l_1^2 -l_1^3)|g(X_{k-1})|^6 h^3 + 12(3-l_1)|g(X_{k-1})|^2 |\langle g(X_{k-1}),X_{k-1} \rangle|^2 h^2}{(D+|X_{k-1}|^2+l_1|g(X_{k-1})|^2 h )^3}.
            \end{align}
       % Next, we use the fundamental inequality $ab \leq a^2/2+b^2 /2$ for all $a$,$b \in \mathbb{R}$ to get
       Then,
       \begin{align}\label{LM310-36}
           B_2 &= \mathbb{E}\Bigg(\frac{3\big[|g(X_{k-1}) \Delta W_{k-1}|^2-l_1|g(X_{k-1})|^2 h    +2 \langle g(X_{k-1}) \Delta W_{k-1},X_{k-1} \rangle \big]^2 \times c_2h}{(D+|X_{k-1}|^2+l_1|g(X_{k-1})|^2 h )^3} \Big | \mathcal{F}_{(k-1)h}\Bigg) \\ \notag
                & = \frac{3(3-2l_1+l_1^2)|g(X_{k-1})|^4 h^2 \times c_2 h +12|\langle g(X_{k-1}),X_{k-1} \rangle|^2 h \times c_2 h}{(D+|X_{k-1}|^2+l_1|g(X_{k-1})|^2 h )^3}.
       \end{align}
        Next,
      \begin{align}\label{LM310-37}
             B_3  &= \mathbb{E}\Bigg(\frac{3\big[|g(X_{k-1}) \Delta W_{k-1}|^2-l_1|g(X_{k-1})|^2 h    +2 \langle g(X_{k-1}) \Delta W_{k-1},X_{k-1} \rangle \big]\times \big(c_2h\big)^2 }{(D+|X_{k-1}|^2+l_1|g(X_{k-1})|^2 h )^3}
            \Big | \mathcal{F}_{(k-1)h}\Bigg) \\  \notag
                & = \frac{3(1 - l_1)|g(X_{k-1})|^2 h \times \big(c_2 h \big)^2}{(D+|X_{k-1}|^2+l_1|g(X_{k-1})|^2 h )^3} .
      \end{align}
       Finally,
      \begin{align}\label{LM310-38}
           B_4   &= \mathbb{E}\Bigg(\frac{ \big(c_2h\big)^3}{(D+|X_{k-1}|^2+l_1|g(X_{k-1})|^2 h )^3}\Big | \mathcal{F}_{(k-1)h}\Bigg) \\ \notag
                & = \frac{\big(c_2 h \big)^3}{(D+|X_{k-1}|^2+l_1|g(X_{k-1})|^2 h )^3}. \notag
      \end{align}
        Since $15-9l_1 + 3l_1^2 -l_1^3 < 0$ and  $1-l_1 < 0$ as well as $3 - l_1 \leq 0 $ when $ l_1 \geq 3 $, by Young’s inequality $|a|^\frac{2-p}{6-p}|b|^{\frac{4}{6-p}} \leq \frac{2-p}{6-p}|a| + \frac{4}{6-p}|b| \leq |a| +|b|$ and $|a||b| \leq |a|^2+|b|^2$ for any $a,b \in \mathbb{R}^d$, then inserting \eqref{LM310-35}- \eqref{LM310-38} into \eqref{LM310-34} leads to
       \begin{align} \label{LM310-39}
               &\quad\mathbb{E}\Big(\xi _{k-1}^3\Big | \mathcal{F}_{(k-1)h}\Big) \\ \notag
               & \leq  \frac{3(3-2l_1+l_1^2)|g(X_{k-1})|^4 c_2 h^3 +12|\langle g(X_{k-1}),X_{k-1} \rangle|^2  c_2 h^2 + (c_2 h)^3}{(D+|X_{k-1}|^2+l_1|g(X_{k-1})|^2 h )^3} \\\notag
               & \leq \frac{1}{(D+|X_{k-1}|^2+l_1|g(X_{k-1})|^2 h )^{p/2}} \Big( \frac{3(3-2l_1+l_1^2)|g(X_{k-1})|^4 c_2 h^3}{D^{1-p/2}\times l_1^2|g(X_{k-1})|^4 h^2} \\ \notag
              & \quad + \frac{12|g(X_{k-1})|^2 |X_{k-1}|^2  c_2 h^2}{D^{1-p/2}\times 4 l_1 |g(X_{k-1})|^2 |X_{k-1}|^2  h } + \frac{(c_2 h)^3}{D^{3-p/2}}\Big)   \\ \notag
              & \leq  \frac{1}{(D+|X_{k-1}|^2+l_1|g(X_{k-1})|^2 h )^{p/2}} \Big(\frac{3 c_2 h}{D^{1-p/2}} + \frac{3 c_2 h}{D^{1-p/2} l_1} + \frac{c_2 ^3 h^3}{D^{3-p/2}} \Big).\notag
       \end{align}
        %then by the fundamental inequality $-(a^2+b^2) \leq -2|a||b|$ for all $a$,$b \in \mathbb{R}$, we note that
    %    \begin{equation}
      %      \begin{aligned}
       %        &\quad(15-9l_1 + 3l_1^2 -l_1^3)|g(X_{k-1})|^6 h^3 + 3(3-2l_1+l_1^2)|g(X_{k-1})|^4 h^2 \times c_2 h  \\
        %       &\quad+ 3(1 - l_1)|g(X_{k-1})|^2 h \times \big(c_2 h \big)^2   \\
         %      & \leq -2 \sqrt{3(1 - l_1)(15-9l_1 + 3l_1^2 -l_1^3)} |g(X_{k-1})|^4 h^2 \times c_2 h + 3(3-2l_1+l_1^2)|g(X_{k-1})|^4 h^2 \times c_2 h  \\
         %      & \leq 0.
          %  \end{aligned}
        %\end{equation}
     %  Furthermore, we also note that
      % \begin{equation}\label{S1-15}
       %    \begin{aligned}
        %       & \quad12|\langle g(X_{k-1}),X_{k-1} \rangle|^2  h \times c_2 h \\
         %      & \leq 12|\langle g(X_{k-1}),X_{k-1} \rangle|| g(X_{k-1})||X_{k-1}|  h \times c_2 h \\
          %     & \leq 12(l_1-3 )|g(X_{k-1})|^2 |\langle g(X_{k-1}),X_{k-1} \rangle|^2 h^2 + \frac{3}{l_1-3}|X_{k-1}|^2 \times (c_2 h)^2
          % \end{aligned}
      % \end{equation}
     %  Therefore, inserting \eqref{S1-10}-\eqref{S1-15} into \eqref{NS7} leads to
     %  \begin{equation}
    %       \begin{aligned}
   %              & \quad \mathbb{E}\Big(\xi _{k-1}^3\Big | \mathcal{F}_{(k-1)h}\Big)  \\
    %              & \leq \frac{3(l_1-3)^{-1}|X_{k-1}|^2 \times (c_2 h)^2 + (c_2 h)^3}{(D+|X_{k-1}|^2+l_1|g(X_{k-1})|^2 h )^3}  \\
     %             & \leq \frac{3}{l_1 - 3}\frac{|X_{k-1}|^2 \times (c_2 h)^2}{ 3 D^2 |X_{k-1}|^2 } + \frac{(c_2 h)^3}{ D^3}  \\
      %            & = \frac{c_2^2 h^2}{l_1 - 3} + \frac{c_2^3 h^3}{ D^3}         \end{aligned}
       %\end{equation}
        Choosing $h$ sufficiently small such that $l_1 h \leq \alpha$, and inserting \eqref{LM310-32},\eqref{LM310-33} and \eqref{LM310-39} into \eqref{LM310-30},  then we have
  \begin{align} \label{LM310-40}
      &\quad\mathbb{E}\Big((D+|X_{k}|^2+l_1|g(X_{k})|^2 h)^{p/2}\Big|\mathcal{F}_{(k-1)h}\Big) \\   \notag
         &\leq\Bigg(1-c_1h\Bigg)^{-p/2}\Bigg(D+|X_{k-1}|^2+l_1|g(X_{k-1})|^2h\Bigg)^{p/2}\Bigg( 1 + \frac{p}{2} \frac{(1-l_1)|g(X_{k-1})|^2 h +c_2h}{D+|X_{k-1}|^2+l_1|g(X_{k-1})|^2 h}
        \\ \notag
         &\quad+\frac{p(p-2)}{8}\frac{4|\langle g(X_{k-1}), X_{k-1} \rangle|^2h}{(D+|X_{k-1}|^2+l_1|g(X_{k-1})|^2 h )^2}  +\frac{p(p-2)(p-4)}{2^3 \times 3!}
       \frac{1}{(D+|X_{k-1}|^2+l_1|g(X_{k-1})|^2 h )^{p/2}}  \\ \notag
       &\quad \times\Big(\frac{3 c_2 h}{D^{1-p/2}} + \frac{3c_2 h}{D^{1-p/2} l_1} + \frac{c_2 ^3 h^3}{D^{3-p/2}} \Big)\Bigg) \\ \notag
         &\leq\Bigg(1-c_1h\Bigg)^{-p/2}\Bigg(D+|X_{k-1}|^2+l_1|g(X_{k-1})|^2h\Bigg)^{p/2}\Bigg( 1+\frac{p}{2} \frac{(1-l_1)|g(X_{k-1})|^2 h}{D+|X_{k-1}|^2+\alpha|g(X_{k-1})|^2 }
         \\ \notag
         &\quad+\frac{p(p-2)}{8}\frac{4|\langle g(X_{k-1}), X_{k-1} \rangle|^2h}{(D+|X_{k-1}|^2+\alpha|g(X_{k-1})|^2 )^2}+\frac{p}{2} \frac{c_2h}{D+|X_{k-1}|^2+l_1|g(X_{k-1})|^2 h }
         \\ \notag
         &\quad+\frac{p(p-2)(p-4)}{2^3 \times 3!} \frac{1}{(D+|X_{k-1}|^2+l_1|g(X_{k-1})|^2 h )^{p/2}}
       \Big(\frac{3 c_2 h}{D^{1-p/2}} + \frac{3c_2 h}{D^{1-p/2} l_1} + \frac{c_2 ^3 h^3}{D^{3-p/2}} \Big)\Bigg) \\ \notag
         &\leq \Bigg(1-c_1h\Bigg)^{-p/2}\Bigg(D+|X_{k-1}|^2+l_1|g(X_{k-1})|^2h\Bigg)^{p/2} \\ \notag
         &\quad \times\Bigg( 1+\frac{ph}{2} \Big(\frac{(1-l_1)|g(X_{k-1})|^2 }{D+|X_{k-1}|^2+\alpha|g(X_{k-1})|^2 }
         -\frac{2|\langle g(X_{k-1}), X_{k-1} \rangle|^2}{(D+|X_{k-1}|^2+\alpha|g(X_{k-1})|^2  )^2}\Big)\\ \notag
         &\quad+ \frac{p^2h}{2} \frac{|\langle g(X_{k-1}), X_{k-1} \rangle|^2}{(D+|X_{k-1}|^2+\alpha|g(X_{k-1})|^2 )^2}
         +\frac{p}{2} \frac{c_2h}{D+|X_{k-1}|^2+l_1|g(X_{k-1})|^2 h } \\ \notag
         &\quad+\frac{p(p-2)(p-4)}{2^3 \times 3!} \frac{1}{(D+|X_{k-1}|^2+l_1|g(X_{k-1})|^2 h )^{p/2}}  \Big(\frac{3 c_2 h}{D^{1-p/2}} + \frac{3c_2 h}{D^{1-p/2} l_1} + \frac{c_2 ^3 h^3}{D^{3-p/2}} \Big)\Bigg).
  \end{align}
  Noting, by \eqref{as33} as well as the elementary inequality $a^2 +b^2 \geq 2|a| |b|$ and $\langle a, b \rangle \leq |a||b|$, we have
  \begin{align} \label{LM310-41}
     &\quad \frac{(1-l_1)|g(X_{k-1})|^2 }{D+|X_{k-1}|^2+\alpha|g(X_{k-1})|^2 }
         -\frac{2|\langle g(X_{k-1}), X_{k-1} \rangle|^2}{(D+|X_{k-1}|^2+\alpha|g(X_{k-1})|^2  )^2}  \leq k_1 + \frac{P(X_{k-1})}{(D+|X_{k-1}|^2+\alpha|g(X_{k-1})|^2  )^2},
  \end{align}
  and
  \begin{align} \label{LM310-42}
       \frac{|\langle g(X_{k-1}), X_{k-1} \rangle|^2}{(D+|X_{k-1}|^2+\alpha|g(X_{k-1})|^2 )^2} &\leq \frac{|X_{k-1} |^2| g(X_{k-1})|^2}{(|X_{k-1}|^2+\alpha|g(X_{k-1})|^2 )^2} \\ \notag
         &\leq\frac{|X_{k-1} |^2| g(X_{k-1})|^2}{(2\sqrt{\alpha}|X_{k-1} || g(X_{k-1})|)^2} \\ \notag
       &=\frac{1}{4\alpha}.
  \end{align}
    If we choose $p$ sufficiently small, then by \eqref{as33} again,
   \begin{align} \label{LM310-43}
        \frac{(D+|X_{k-1}|^2+l_1|g(X_{k-1})|^2h)^{p/2}P(X_{k-1})}{(D+|X_{k-1}|^2+\alpha|g(X_{k-1})|^2  )^2} \leq C_3.
   \end{align}
 Therefore, Substituting \eqref{LM310-41}, \eqref{LM310-42} and \eqref{LM310-43} into \eqref{LM310-43}, let $h \leq 1/2c_1$, we obtain
   \begin{align} \label{LM310-44}
        &\quad\mathbb{E}\Big((D+|X_{k}|^2+l_1|g(X_{k})|^2 h)^{p/2}\Big|\mathcal{F}_{kh}\Big) \\  \notag
          &\leq \Bigg(1-c_1h\Bigg)^{-p/2}\Bigg(D+|X_{k-1}|^2+l_1|g(X_{k-1})|^2h\Bigg)^{p/2}\Bigg( 1+\frac{p}{2}k_1h+\frac{p^2}{2}\frac{h}{4\alpha}\Bigg)   \\ \notag
         &\quad+\Bigg(1-c_1h\Bigg)^{-p/2}\Bigg(\frac{ph}{2}\frac{(D+|X_{k-1}|^2+l_1|g(X_{k-1})|^2)^{p/2}P(X_{k-1})}{(D+|X_{k-1}|^2+\alpha|g(X_{k-1})|^2  )^2}h +\frac{p}{2} \frac{c_2h}{D^{1-p/2}}  \\ \notag
         & \quad+\frac{p(p-2)(p-4)}{2^3 \times 3!}
          \Big(\frac{3 c_2 h}{D^{1-p/2}} + \frac{3c_2 h}{D^{1-p/2} l_1} + \frac{c_2 ^3 h^3}{D^{3-p/2}} \Big)\Bigg).      \\ \notag
            &\leq \Bigg(1-c_1h\Bigg)^{-p/2}\Bigg(D+|X_{k-1}|^2+l_1|g(X_{k-1})|^2h\Bigg)^{p/2}\Bigg( 1+\frac{p}{2}k_1h+\frac{p^2h}{8\alpha}\Bigg)+ Mh,
   \end{align}
    where $M$ is a positive constant depends on $D, C_3, c_2$ and $p$. Furthermore, for any $p \in (0,1)$, $h \in (0,1/c_1)$ and $\kappa \in [-\frac{1}{2},\frac{1}{2}]$, it is obvious that
    \begin{equation}\label{LM310-45}
        (1-c_1h)^{p/2}\geq 1-\frac{p}{2}c_1h-K_1h^2
    \end{equation}
    as well as
    \begin{equation}\label{LM310-46}
         \frac{1}{1-\kappa} \leq 1+\kappa+\kappa^2\sum\limits _{i=0}^\infty(\frac{1}{2})^i=1+\kappa+2\kappa^2,
    \end{equation}
    where $K_1=K_1(c_1,p)$ is a positive constant. Define $\epsilon :=\frac{1}{4}|k_1+c_1|$.  And if necessary, we  choose $h^*$ and $p^*$ suﬀiciently small such that for any $h \in (0, h^*]$ and $p \in (0, p^*]$,
    \begin{equation}\label{LM310-47}
        \frac{p^\ast}{4\alpha} \leq \epsilon,
    \end{equation}
    \begin{equation} \label{LM310-48}
       \left|\frac{p}{2}c_1h+K_1h^2\right| \leq \frac{1}{2},
   \end{equation}
   and
   \begin{equation}  \label{LM310-49}
       K_1h+2\bigg(\frac{p}{2}c_1+K_1h\bigg)^2h+\frac{p}{2}\bigg(k_1+\epsilon\bigg)\bigg((\frac{p}{2}c_1h+K_1h^2)
             +2(\frac{p}{2}c_1h+K_1h^2)^2\bigg)
             \leq \frac{p}{2}\epsilon.
   \end{equation}
   Therefore, combining \eqref{LM310-44}-\eqref{LM310-49} together,  we obtain that
   \begin{equation*}
         \begin{aligned}
             &\mathbb{E}\bigg(D+|X_{k}|^2+l_1h|g(X_{k})|^2 \bigg)^{p/2}   \\
            \leq & \bigg(\frac{1+\frac{p}{2}k_1h+\frac{p}{2}\epsilon h}{1-(\frac{p}{2}c_1h+K_1h^2)} \bigg)\mathbb{E}\bigg(D+|X_{k-1}|^2+l_1h|g(X_{k-1})|^2 \bigg)^{p/2}+Mh   \\
             \leq& \bigg(1+\frac{p}{2}k_1h+\frac{p}{2}\epsilon h\bigg)\bigg(1+(\frac{p}{2}c_1h+K_1h^2)+2(\frac{p}{2}c_1h+K_1h^2)^2 \bigg)\mathbb{E}\bigg(D+|X_{k-1}|^2 +l_1h|g(X_{k-1})|^2 \bigg)^{p/2}+Mh      \\
             \leq & \bigg(1+\frac{p}{2}(k_1+c_1+\epsilon)h+\frac{p}{2}\epsilon h\bigg)\mathbb{E}\bigg(D+|X_{k-1}|^2+l_1h|g(X_{k-1})|^2 \bigg)^{p/2}+Mh    \\
             \leq &\bigg(1+\frac{p}{4}(k_1+c_1)h\bigg)\mathbb{E}\bigg(D+|X_{k-1}|^2+l_1h|g(X_{k-1})|^2 \bigg)^{p/2}+Mh.
             \end{aligned}
     \end{equation*}
     Let $h < -4/(p(k_1+c_1))$. By iteration, we get
     \begin{align*}
         &\mathbb{E}\bigg(D+|X_{k}|^2+l_1h|g(X_{k})|^2 \bigg)^{p/2} \\
         \leq& \bigg(1+\frac{p}{4}(k_1+c_1)h\bigg)^k\mathbb{E}\bigg(D+|x_0|^2+l_1h|g(x_0)|^2 \bigg)^{p/2}+\frac{1-(1+\frac{p}{4}(k_1+c_1)h)^k}{1-(1+\frac{p}{4}(k_1+c_1)h)}Mh.
      \end{align*}
     This implies that
     \begin{equation*}
         \mathbb{E}(|X_{k}|^p )\leq \mathbb{E}\bigg(D+|x_0|^2+l_1h|g(x_0)|^2 \bigg)^{p/2}-\frac{4M}{p(k_1+c_1)}.
     \end{equation*}
      The proof is completed.
\end{proof}

\begin{lemma}\label{LM311}
    Suppose Assumption \eqref{as32} and \eqref{as33} hold, then there exists a constant $p^* \in (0,1)$ such for any $p \in (0, p^*]$ and $D \in (0, \infty)$, the solution of \eqref{SDE} satisfies
    \begin{equation*}
    \mathbb{E}|x(t)|^p\leq \mathbb{E}\big(D+|x_0|^2\big)^\frac{p}{2}+K_2
\end{equation*}
for any $t > 0$, where $K_2$ is a constant depends on $p,c_1,c_2$ and $D$.
\end{lemma}
\begin{proof}
    First,  Let $\epsilon := p|c_1+k_1|/4$, by \eqref{SDE}, \eqref{as33} and the It\^{o} formula, we derive that
    \begin{equation}
\begin{aligned}\label{LM311-50}
    &\quad\mathbb{E}\Big(e^{\epsilon t}(D+|x(t)|^2)^\frac{p}{2}\Big) \\
    &\leq \mathbb{E}\Big(D+|x_0|^2\Big)^\frac{p}{2}+\mathbb{E}\int_0^t \Bigg(\epsilon e^{\epsilon s }\Big(D+|x(s)|^2\Big)^\frac{p}{2}+\frac{p}{2}e^{\epsilon s }\Big(D+|x(s)|^2\Big)^{\frac{p}{2}-1}      \\
    & \quad\Big(2\langle x(s),f(x(s))\rangle+|g(x(s))|^2 + \frac{(p-2)|\langle x(s),g(x(s))\rangle|^2}{D+|x(s)|^2}\Big)\Bigg)ds
     \end{aligned}
\end{equation}
where
\begin{align} \label{LM311-51}
    &\quad2\langle x(s),f(x(s))\rangle+|g(x(s))|^2+ \frac{(p-2)|\langle x(s),g(x(s))\rangle|^2}{D+|x(s)|^2}  \\ \notag
        &=\bigg(2\langle x(s),f(x(s))\rangle+l_1|g(x(s))|^2 \bigg)+\bigg(D+|x(s)|^2\bigg)\bigg(\frac{(1-l_1)|g(x(s)))|^2}{D+|x(s)|^2} + \frac{(p-2)|\langle x(s),g(x(s))\rangle|^2}{(D+|x(s)|^2)^2}\bigg) \\ \notag
        &\leq c_2+c_1|x(s)|^2+\bigg(D+|x(s)|^2\bigg)\bigg(\frac{(1-l_1)|g(x(s)))|^2}{D+|x(s)|^2+\alpha|g(x(s)|^2} + \frac{(p-2)|\langle x(s),g(x(s))\rangle|^2}{(D+|x(s)|^2+\alpha|g(x(s))|^2)^2}\bigg) \\ \notag
        &\leq c_2+c_1|x(s)|^2+\bigg(D+|x(s)|^2\bigg)\bigg(k_1+\frac{P(x(s))}{(D+|x(s)|^2+\alpha|g(x(s))|^2)^2} +\frac{p|\langle x(s),g(x(s))\rangle|^2}{(D+|x(s)|^2+\alpha|g(x(s))|^2)^2}\bigg)\\ \notag
        &\leq c_2+\bigg(D+|x(s)|^2\bigg)\bigg(c_1+k_1+\frac{P(x(s))}{(D+|x(s)|^2+\alpha|g(x(s))|^2)^2} +\frac{p|\langle x(s),g(x(s))\rangle|^2}{(D+|x(s)|^2+\alpha|g(x(s))|^2)^2}\bigg).
\end{align}
Following the same arguments used in the derivation of \eqref{LM310-44}, we can get
       \begin{align}\label{LM311-52}
            \bigg(D+|x(s)|^2\bigg)^{\frac{p}{2}-1} \bigg(D+|x(s)|^2\bigg)\bigg(\frac{P(x(s))}{(D+|x(s)|^2+\alpha|g(x(s))|^2)^2}\bigg) \leq  C_3 ,
       \end{align}
and
\begin{equation}\label{LM311-53}
    \frac{p|\langle x(s),g(x(s))\rangle|^2}{(D+|x(s)|^2+\alpha|g(x(s))|^2)^2} \leq \frac{p}{4\alpha},
\end{equation}
where $C_3$ is a constant specified in \eqref{as33}. Let $p$ sufficiently small such that $p/4\alpha \leq |c_1+k_1|/2$, then we substitute \eqref{LM311-51}-\eqref{LM311-53} into \eqref{LM311-50}
\begin{equation}\label{LM311-54}
    \begin{aligned}
        \mathbb{E}\Big(e^{\epsilon t}(D+|x(t)|^2)^\frac{p}{2}\Big) & \leq  \mathbb{E}\Big(D+|x_0|^2\Big)^\frac{p}{2}+\int_0^t \frac{p}{2} e^{\epsilon s}\Big(\frac{c_2}{D^{1-\frac{p}{2}}}+C_3\Big)ds  \\
        & \leq  \mathbb{E}\Big(D+|x_0|^2\Big)^\frac{p}{2}+\frac{p}{2}\frac{1}{\epsilon}\Big(e^{\epsilon t}-1\Big)\Big(\frac{c_2}{D^{1-\frac{p}{2}}}+C_3\Big).
    \end{aligned}
\end{equation}
Let $K_2 := (p/(2\epsilon))((c_2/D^{1-p/2})+C_3)$, and then we divide both sides of \eqref{LM311-54} by $e^{\epsilon t}$ such that
\begin{equation*}
      \mathbb{E}\Big((D+|x(t)|^2)^\frac{p}{2}\Big)
    \leq \mathbb{E}\Big(D+|x_0|^2\Big)^\frac{p}{2} \times e^{\frac{p}{4}(c_1+k_1)t} + K_2.
\end{equation*}
Since $c_1+k_1 < 0$, then the desired result follows.
\end{proof}

\subsection{The global attractivity.}
In this subsection, we will show the global attractivity of the underlying solution under the above assumptions.
\begin{lemma}\label{LM312}
    Suppose \eqref{as32} and \eqref{as34} hold, then there exists a sufficient small constant $p \in (0, 1)$  such that for any $t > 0$, the solution of \eqref{SDE} satisfies
 \begin{equation*}
    \mathbb{E}|x(t,x_0)-x(t, y_0)|^p \leq \mathbb{E}(|x_0-y_0|^p)e^{\frac{p}{4}(c_3+k_2)t}.
    \end{equation*}
\end{lemma}
\begin{proof}
 First of all, we define $D(t):=x(t, x_0)-x(t, y_0),G(t):=g(x(t, x_0))-g(x(t, y_0)),F(t):=f(x(t, x_0))\\
 -f(x(t, y_0))$, and  $\gamma := p|c_3 + k_2|/4$, then by It\^o formula
  \begin{align}\label{LM312-55}
       \mathbb{E}\Big(e^{\gamma t}|D(t)|^p\Big) &=  \mathbb{E}\Big(|D(0)|^p\Big)+\mathbb{E}\int_0^t \gamma e^{\gamma s} |D(s)|^p+\frac{p}{2}e^{\gamma s} \Big(|D(s)|^2\Big)^{\frac{p}{2}-1}\\ \notag
    & \quad \times \Big(2\langle D(s),F(s)\rangle+|G(s)|^2 +\frac{(p-2)|\langle D(s),G(s)\rangle|^2}{|D(s)|^2}\Big)ds.    \notag
  \end{align}
Since $l_2 > 1$ and $p \in (0,1)$, we have
\begin{align*}
      &\quad2\langle D(s),F(s)\rangle+|G(s)|^2 +\frac{(p-2)|\langle D(s),G(s)\rangle|^2}{|D(s)|^2}  \\
        &=\bigg(2\langle D(s), F(s)\rangle+l_2|G(s)|^2\bigg)+\bigg(|D(s)|^2\bigg)\bigg(\frac{(1-l_2)|G(s)|^2}{|D(s)|^2}+\frac{(p-2)|\langle D(s),G(s)\rangle|^2}{(|D(s)|^2)^2} \bigg)   \\
        &\leq  \bigg( c_3|D(s)|^2\bigg)+\bigg(|D(s)|^2\bigg)\bigg(\frac{(1-l_2)|G(s)|^2}{|D(s)|^2+\beta|G(s)|^2}+\frac{(p-2)|\langle D(s),G(s)\rangle|^2}{(|D(s)|^2+\beta|G(s)|^2)^2} \bigg)   \\
        &\leq  \bigg( c_3|D(s)|^2\bigg)+\bigg(|D(s)|^2\bigg)\bigg(\frac{(1-l_2)|G(s)|^2}{|D(s)|^2+\beta|G(s)|^2}-\frac{2|\langle D(s),G(s)\rangle|^2}{(|D(s)|^2+\beta|G(s)|^2)^2}+\frac{p|\langle D(s),G(s))\rangle|^2}{(|D(s)|^2+\beta|G(s)|^2)^2}\bigg),
\end{align*}
Similar to the derivation of \eqref{LM310-42}, we can see that
\begin{equation*}
    \frac{p|\langle D(s),G(s)\rangle|^2}{(|D(s)|^2+\beta|G(s)|^2)^2} \leq \frac{p|D(s)|^2|G(s)|^2}{(2\sqrt{\beta}|D(s)||G(s)|)^2} =\frac{p}{4\beta}
\end{equation*}
 Let $p=2\beta|c_3+k_2|$, by \eqref{as34}, we have
\begin{equation} \label{LM312-56}
    \begin{aligned}
        2\langle D(s),F(s)\rangle+|G(s)|^2 +\frac{(p-2)|\langle D(s),G(s)\rangle|^2}{|D(s)|^2}
        &\leq c_3|D(s)|^2+|D(s)|^2(k_1+\frac{p}{4\beta} ) \\
        &\leq\frac{1}{2}(c_3+k_2)(|D(s)|^2),
    \end{aligned}
\end{equation}
Inserting \eqref{LM312-56} into \eqref{LM312-55} yields
\begin{align*}
     \mathbb{E}\Big(e^{\gamma t}|D(t)|^p\Big)
       &\leq  \mathbb{E}(|D(0)|^p + \mathbb{E}\int_0^t \gamma e^{\gamma s} |D(s)|^p+\frac{p}{2}e^{\gamma s}\Big(|D(s)|^2\Big)^{\frac{p}{2}-1}\Big(\frac{1}{2}(c_3+k_2)(|D(s)|^2)\Big) ds \\
       & \quad= \mathbb{E}(|D(0)|^p + \mathbb{E}\int_0^t (\gamma + \frac{p(c_3+k_2)}{4})e^{\gamma s} |D(s)|^p ds \\
       &\quad=\mathbb{E}(|D(0)|^p
\end{align*}
Dividing both sides by $e^{\gamma t}$ gives
\begin{equation*}
    \begin{aligned}
        \mathbb{E}(|D(t)|^p)\leq  \mathbb{E}(|D(0)|^p)e^{\frac{p}{4}(c_3+k_2)t}.
    \end{aligned}
\end{equation*}
Thus, the proof is completed.
\end{proof}
\subsection{Convergence}
First, we will show the finite-time convergence result that we need. In fact, this result has been proved in \cite{CG2020}, but we need to make a little modification to meet our requirements.
\begin{lemma}\label{LM36}
    Suppose \eqref{as31} and \eqref{as32} hold with $l_1 \geq 2q$, then there exists a pair of constants $(p^*, h^*)$ consistent with Lemmas \ref{LM310} and \ref{LM311} such that for any $p \in (0, \frac{2p^*}{l_1}]$ and $h \in (0, h^*]$, the solution of \eqref{SDE} and the solution of BEM method \eqref{DASDE} satisfie
\begin{align*}
   \sup \limits _{0\leq k \leq n} \mathbb{E}\Big(|x(kh) - X_k|^{p}\Big) \leq C_T\Big(1+\mathbb{E}(|x(0)|^{\frac{l_1 p}{2}})\Big)h^{\frac{p}{2}}
\end{align*}
for any $T > 0$, where $n:=T/h$ and $C_T$ is a positive constant depends on $T$.
\end{lemma}
\begin{proof}
    When the initial value $x(0)$ of \eqref{SDE} is constant, we state the convergence result in \cite{CG2020} as follows.
    \begin{align*}
        \mathbb{E}\Big(|x(kh) - X_k|^{2}\Big) \leq C_T|x(0) - X_0|^{2} + C_T(1+|x(0)|^{l_1})h.
    \end{align*}
  And if the initial value $x(0)$ is a random variable, we can easily get that
     \begin{align*}
        \mathbb{E}\Big(|x(kh) - X_k|^{2}\Big|  \mathcal{F}_{0}\Big) \leq C_T|x(0) - X_0|^{2} + C_T(1+|x(0)|^{l_1})h.
    \end{align*}
    Then by the H\"{o}lder inequality and the elementary inequality \eqref{ELE}, we have
    \begin{align*}
        \mathbb{E}\Big(|x(kh) - X_k|^{p} \Big| \mathcal{F}_{0} \Big) &\leq
        \Bigg( \mathbb{E}\Big(|x(kh) - X_k|^{2}\Big|  \mathcal{F}_{0}\Big)\Bigg)^{\frac{p}{2}} \\ \notag
        &\leq
        \Bigg( C_T|x(0) - X_0|^{2} + C_T(1+|x(0)|^{l_1})h\Bigg)^{\frac{p}{2}} \\ \notag
        & \leq  C_T|x(0) - X_0|^{p} + C_T(1+(|x(0)|^{\frac{l_1 p}{2}}))h^{\frac{p}{2}},
    \end{align*}
    this implies
    \begin{align*}
        \mathbb{E}\Big(|x(kh) - X_k|^{p}\Big) \leq C_T\mathbb{E}\Big(|x(0) - X_0|^{p}\Big) + C_T\Big(1+\mathbb{E}(|x(0)|^{\frac{l_1 p}{2}})\Big)h^{\frac{p}{2}}.
    \end{align*}
    The required assertion follows.
\end{proof}

From \eqref{LM311}-\eqref{LM36}, let $r_1 = l_1p/2$, $p=p$, and $\Psi(x)=1+x^{\frac{l_1 p}{2}}$, Condition \ref{as21} and \ref{as22} are satisfied, then by \eqref{T22} and \eqref{T25}, we can conclude this part by the following theorems.
\begin{theorem}\label{T312}
    Under \eqref{as31}-\eqref{as34}, the BEM solution converges uniformly to the underlying solution of \eqref{SDE}.
\end{theorem}
Next, since \eqref{T312} and \eqref{as22} hold, from \eqref{T25}, we get the last theorem.
\begin{theorem}\label{T315}
   Under \eqref{as31}-\eqref{as34}, the probability measure of the BEM solution $\{X_k\}_{k \geq 0}$ converges to the underlying stationary distribution $\pi (\cdot)$.
\end{theorem}

\section{Numerical examples}
\begin{example}
    The Ginzburg-Landau equation is from the theory of superconductivity. Its
stochastic version with multiplicative noise can be written as
\begin{equation}\label{NE-1}
    dx(t) = \Big((\alpha + \frac{1}{2} \sigma^2)x(t) - x^3(t)\Big)dt + \sigma x(t)dW(t),
\end{equation}
And the exact solution is known to be \cite{KP1992}
\begin{equation*}
    x(t) = \frac{x(0)\exp(\alpha t + \sigma W(t))}{\sqrt{1+2x^2(0)\int_0^t \exp (2 \alpha s + 2 \sigma W(s))ds}}.
\end{equation*}
If $\alpha = -\frac{1}{4}$ and $\sigma = 1$, by setting $l_1 = 10$, $l_2=3$, and $\alpha = \beta =0.01$, it is not hard to see that the
drift and diffusion coefficients of \eqref{NE-1} satisfy \eqref{as31} - \eqref{as34} with $p = 0.001 $, $q=3$, $L_1=1.5$, $k_1=-10.75$, $c_1=10.5$, $c_2=0$, $k_2=-3.75$ and $c_3=3.5$.
From our previous analysis, we can see that the BEM method is strongly uniform convergence with order p, which means that for any $j>0$ and $h \in (0, h^*)$, there is a constant $C>0$ such that
\begin{equation}\label{NE-2}
  e^{strong}_{h} := \mathbb{E}|x(jh) - X_j|^p \leq C h^{\frac{p}{2}}.
\end{equation}
Then we use the BEM method to simulate 1000 sample paths with $x_0 = 1$ and $h = 0.001$, and the mean of sample points generated by these paths at the same time point are used to construct the approximation to $e^{strong}_{h}$ against $h^{\frac{p}{2}}$ over $[0, 200]$. As shown in Figure \ref{estrong}, it is clear that the curve is roughly stable, which indicates that there is indeed a constant $C$ satisfying the inequality \eqref{NE-2}.

\begin{figure}
    \centering
    \includegraphics[width=0.80\textwidth]{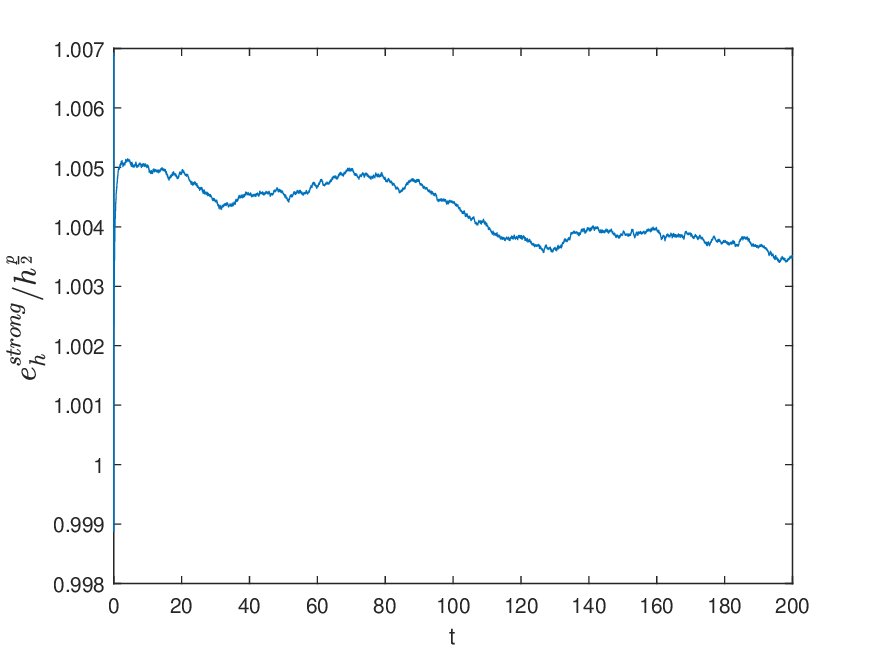}
    \caption{The approximation to $e^{strong}_{h}$ against $h^{\frac{p}{2}}$ over $[0, 200]$}
    \label{estrong}
\end{figure}
\end{example}

\begin{example}
    Consider a two dimensional SDE
      \begin{align*}\label{ex2}
             d\left[
        \begin{array}{cc}
             x_1(t) \\
             x_2(t)
        \end{array}
        \right]
        &=
        \left[
        \begin{array}{cc}
         1+   0.1x_1(t) -x_2(t) - 21x_1^3(t) -21x_1^5(t) \\
          1+   x_1(t)+ 0.1x_2(t) - 21x_2^3(t) -21x_2^5(t)
        \end{array}
        \right]
        dt  \\
        & \quad+
         \left[
        \begin{array}{cc}
            x_1(t) -0.2x_2(t) + x_1^3(t) - 0.2 x_2^3(t)\\
             0.2x_1(t) + x_2(t) + 0.2 x_1^3(t) + x_2^3(t)
        \end{array}
        \right]
        dW(t).
      \end{align*}

     Let $l_1 =20$, $l_2 = 5$ and $\alpha = \beta =0.025$, it is not hard to see that \eqref{as31} -\eqref{as34} are satisfied with $p = 0.001 $, $q=5$, $L_1=40$, $k_1=-21$, $c_1=20.5$, $c_2=10$, $k_2=-6$ and $c_3=5.5$.
We use the BEM method to simulate 1000 sample paths with $x_0=[1,1]^T$ and $h=0.001$, and then the sample points generated by these paths at the same time point are used to construct the corresponding empirical density function.
According to the \eqref{T315}, the underlying solution has a unique stationary distribution. However, its explicit form is hard to find.
Therefore, to intuitively show that the underlying solution does have a unique stationary distribution, the empirical distribution at $t=5$ is regarded as the stationary distribution. Then we use the Kolmogorov-Smirnov test (K-S test) to measure the difference between the empirical distribution and the stationary distribution at each time point. As shown in Figure \ref{ks-test}, the difference tends to 0, which indicates the numerical stationary distribution is quite a good approximation to the underlying one.
 \begin{figure}[H]
    \centering
    \includegraphics[width=0.80\textwidth]{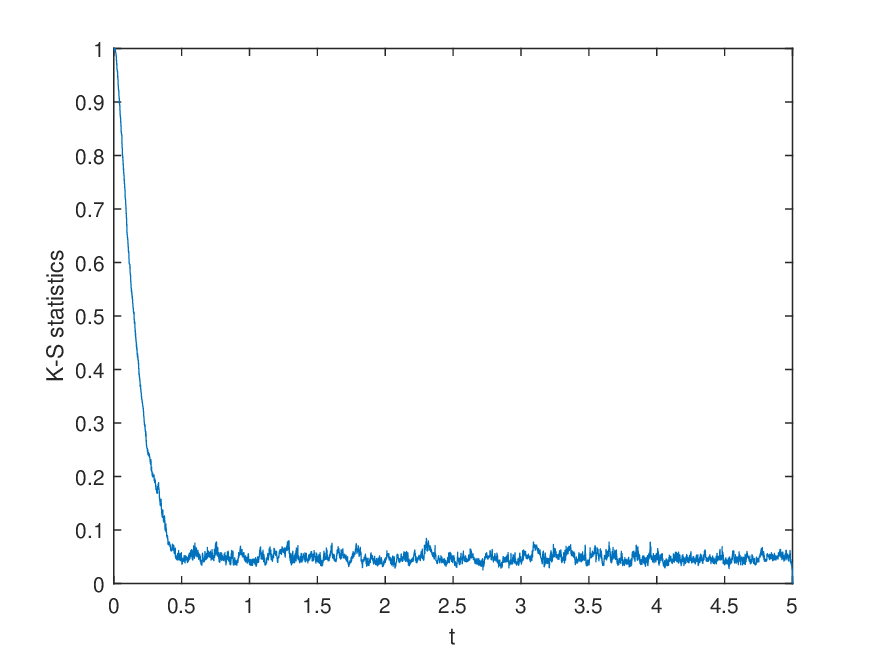}
    \caption{K-S statistics}
    \label{ks-test}
\end{figure}
\end{example}

\section{Conclusion and future works}
In this paper, a quite general result about the strong convergence in the infinite horizon of numerical methods for SDEs is proved. This result could cover many different numerical methods, as the proof does not need the detailed structure of the numerical methods.
\par
In addition, as the driven noise are only required to be the independent and stationary, the results still hold if the Brownian motion is replaced by other proper processes.
\par
Right now, we are working on a similar result for stochastic delay differential equation (SDDEs) and trying to build a connection between the numerical methods for SDEs and SDDEs.

\bibliography{References}

\end{document}